\documentclass[12pt]{amsart} %bkl6 changes: Radha 22.12.2022 Dave 26.12.2022 Markus 8.1.2022
\usepackage{amssymb}
\usepackage[all]{xy}
\usepackage{fullpage}
\usepackage{diagbox}
\usepackage[pdfencoding=auto]{hyperref}
\allowdisplaybreaks

\numberwithin{equation}{section}
\theoremstyle{plain}
\newtheorem{theorem}[equation]{Theorem}
\newtheorem{lemma}[equation]{Lemma}

\theoremstyle{definition}

\theoremstyle{remark} 
\newtheorem{remark}[equation]{Remark}

\newcommand{\bZ}{\mathbb{Z}}

\renewcommand{\ge}{\geqslant}
\renewcommand{\geq}{\geqslant}
\newcommand{\HH}{H\!H}

\renewcommand{\leq}{\leqslant}

\newcommand{\sym}{\mathfrak S}%symmetric group

\author{David Benson, Radha Kessar, and Markus Linckelmann}

\address{David Benson \\ 
Institute of Mathematics\\ 
Fraser Noble Building\\
University of Aberdeen\\ 
King's College\\ 
Aberdeen AB24 3UE\\ 
United Kingdom}

\address{Radha Kessar \\
Department of Mathematics\\
University of Manchester\\
Oxford Road\\
Manchester M193PL\\
United Kingdom}

\address{Markus Linckelmann \\
School of Science \& Technology \\
Department of Mathematics \\
City, University of London \\
Northampton Square \\
London EC1V 0HB \\
United Kingdom}



\title{Hochschild cohomology of symmetric groups and generating
  functions, II}

\begin{document}

\begin{abstract}
We  relate   the generating functions  of the dimensions of the Hochschild
cohomology in any fixed degree of  the symmetric groups   with
those of blocks of the symmetric groups.   We show that the first Hochschild
cohomology of  a positive defect block of a symmetric group is
non-zero, answering  in the
affirmative a question  of the third author. To do this, we prove a
formula expressing the dimension of degree one Hochschild cohomology as a sum of
dimensions of centres of blocks of smaller symmetric groups. This in turn is a
consequence of a general formula that makes more precise a
theorem of our previous paper describing the 
generating functions for the dimensions of Hochschild
cohomology  of symmetric groups.
\end{abstract}

\maketitle

\section{Introduction}
Let $ p$ be a prime number and $k$ a field of characteristic $p$.        
As a consequence of results in \cite{Fleischmann/Janiszczak/Lempken:1993a}, using the              
classification of finite simple groups, if $G$ is a finite group of order divisible by $p$, then
$\HH^1(kG)$ is non-zero.
 It is an open question  
\cite[Question 7.4]{Linckelmann:2018b}  whether for $G$ a finite group and $B$ a block of $kG$,  
if the  defect groups of $B$
are non-trivial, then   $\HH^1(B)  $ is non-zero. This question has been shown to have a positive
answer in some cases in \cite{Todea:2023a} and  \cite{Murphy:2023a}, for instance.
We prove that this question has an affirmative answer  if $G$ is a
symmetric group $\sym_n$ on $n$ letters.

\begin{theorem}\label{HH1-blocknonvanish}
Let $B$ a  block of $k\sym_n $   with  non trivial  defect groups.  Then $\HH^1(B)\ne 0 $.
\end{theorem}  

In fact, we give a  precise formula for the dimension  of
the first  Hochschild cohomology of a block of a symmetric group  as a
sum of  dimensions of  the centres of blocks of smaller symmetric
groups (Theorem~\ref{HH1-blockdim}), and this easily implies
Theorem~\ref{HH1-blocknonvanish}.

In order to state the formula, let us recall that to each   block $B$
of $k\sym_n $ is  associated a non-zero  integer  $w \leq \lfloor
n/p\rfloor$  called the  weight  of $B$, with the property that   the
Sylow $p$-subgroups   of $\sym_{pw} $ are defect  groups of  $B$ under
the  natural inclusion $\sym_{pw} \leq \sym_n$.   In particular, $B$
has  non-trivial defect groups if and only if  $w>0$. Moreover, by
Theorem~7.2 of Chuang and Rouquier \cite{Chuang/Rouquier:2008a} if $B$
and $B'$ are blocks of possibly different symmetric groups, with the
same weight, then  $B$ and $B'$  are derived
equivalent algebras, and consequently,  for any  $r\geq 0 $, we  have
$\dim_k\HH^r (B)  = \dim_k\HH^r (B')$.

For  $w \geq 0 $, denote by $B_{pw}$  the principal  block  of
$k\sym_{pw} $.  Then $B_{pw} $   has weight $w$  and  by the above, 
$\dim_k\HH^r (B)  =  \dim_k\HH^r (B_{pw})$ for any  weight $w$
block $B$ of  a symmetric group algebra.    Thus
Theorem~\ref{HH1-blocknonvanish}  is   a consequence of the following
result.    For $w \geq 0  $, let $\rho(pw, \varnothing) $ equal
the number of  partitions of $pw$ with empty $p$-core.

\begin{theorem}\label{HH1-blockdim} 
Let    $B$  be a  weight $w$-block of    a symmetric group  algebra over $k$.  If $p=2 $, then 
\[  \dim_k \HH^1(B)  =   2\sum_{j=0}^{w-1} \dim_k  Z(B_{pj} )  = 2
 \sum_{j=0}^{w-1}\rho(pj, \varnothing). \]
If $p \geq 3 $, then
\[ \dim_k \HH^1(B)  =   \sum_{j=0}^{w-1} \dim_k  Z(B_{pj} )  =
  \sum_{j=0}^{w-1}\rho(pj, \varnothing). \]
\end{theorem}  

The proof of the  above formula  goes  through the  following
theorem  relating generating functions of dimensions of Hochschild
cohomology of blocks of symmetric groups with those of the entire
group algebra and then  invoking  the results of  our previous paper
\cite{Benson/Kessar/Linckelmann:bkl5}.  
Denote by $p(n)$ the number of partitions of $n$, and by $P(t)$ 
the generating function $\sum_{n=0}^\infty p(n)t^n$.
Note that $\dim_k\HH^0(k\sym_n )=\dim_kZ(k\sym_n)=p(n) $.

\begin{theorem} \label{HH-group-block}     
Set  $\displaystyle  Z(t)= \sum_{n=0}^\infty \dim_kZ(B_{pn})\,t^n$.
For any $r\geq 1 $, there exists a rational function $\phi(t)$ 
(depending on $p$ and $r$)  with $\phi(0)$ non-zero, such that
\[ \sum_{n = 0}^\infty \dim_k\HH^r(B_{pn})\,  t^{n} =
  t \phi(t)  Z(t) \] 
and 
\[ \sum_{n = 0}^\infty \dim_k\HH^r( k\sym_n)\,  t^{n}
= t^p \phi(t^p) P(t). \]
\end{theorem} 

\begin{remark}
In \cite{Benson/Kessar/Linckelmann:bkl5}, we proved that 
\[ \sum_{n = 0}^\infty \dim_k\HH^r(B_{pn})\,  t^{n} =
R_{p,r}(t)P(t) \]
where $R_{p,r}(t)$ is a rational function of $t$. Ken Ono asked us
whether $R_{p,r}(t)$ is a rational function of $t^p$, and
Theorem~\ref{HH-group-block} proves that this is the case,
with $R_{p,r}(t)=t^p\phi(t^p)$.
\end{remark}

\begin{remark} \label{rem2} 
The constant coefficient   of $\phi(t)$ in Theorem \ref{HH-group-block} 
is equal to $y_1=\dim_k(\HH^r(B_p))$. Since $B_p$ is derived equivalent to 
$k(C_p\rtimes C_{p-1})$, we have
$y_1=\dim_k \HH^r(k(C_p\rtimes C_{p-1}))$. An easy
calculation, using the centraliser decomposition, shows that for $r\geq 1$ we have 
$y_1=2$ if $r \equiv 0$ or $-1$ modulo $2(p-1)$ (which is
in particular the case if $p=2$) and $y_1=1$ otherwise.
\end{remark} 

\begin{remark} \label{rem3} 
We note that the above results do not depend on the choice of the field $k$.
If $k'$ is an extension field of $k$ and $B$ a block of $k\sym_n$ for some positive
integer $n$, then $B'=k'\otimes_k B$ is a block of $k'\sym_n$ having the same defect 
groups as $B$, and for any finite-dimensional $k$-algebra $A$  we have  a graded
$k$-algebra isomorphism
$\HH^*(k'\otimes_k A)\cong$ $k'\otimes_k \HH^*(A)$.
\end{remark} 

\section{Proofs.}  %%%%%%%%%%%%%%%%%%%%%%%%%%%%%%%%%

We begin with an elementary lemma.

\begin{lemma} \label{p-subs}    
Let $R$ be an integral domain  and $m$ a positive integer.  If $ h(t)
\in R[[t]] $  is such that  $h (t^m)  \in R((t))  $ is a rational
function,  then $h(t) $ is also a rational function.
\end{lemma} 

\begin{proof}    
Let $h(t) = \sum_{n=0}^{\infty} h_n t^n $, $ h_n \in
  R$    and suppose that  $a(t), b(t)  \in  R[t]$   are such that
\[  \sum_{n=0}^{\infty} h_n t^{mn} =   h(t^m)  = \frac{a(t)}{b(t)} . \]
If $h(t) =0$, then there is nothing to prove.  So, we assume that
$h(t) \ne 0$.  Write 
\[ a(t)  = \sum_{s=0}^{m-1}    a_s(t^m)\,  t^s,    \qquad b(t) =
  \sum_{s=0}^{m-1}    b_s(t^m)\,  t^s, \]
for    $a_s(t), b_s(t)  \in R[t]$, $ 0\leq s \leq  m-1 $. 
Comparing coefficients of powers of $t$, the equality  
\[ a(t)  =  b(t) \,   \sum_{n=0}^{\infty} h_n\, t^{mn}  \]
implies the  equality    
\[ b_s(t^m)  =     a_s(t^m)\,\sum_{n=0}^{\infty} h_n\, t^{mn}  \] 
for each $ s$, $0\leq s \leq m-1$.
Choose $ s  $ such that $a_s (t)  \ne 0 $. Then $h(t)
=a_s(t) /b_s(t) $  is a rational  function of $t$.
\end{proof}

\begin{proof}[Proof of Theorem~\ref{HH-group-block}]
Let $r \geq 1 $.  For $ n\geq 0 $, let $c(n) $  denote   the number
of $p$-core partitions of $n$ and for  each $s$,    $ 0\leq s \leq
p-1$, set   $C_s(t)  =  \sum_{n=0} ^{\infty}  c(np+s) \, t^n $.  
For $w \geq 0 $, set   $z_{pw} =  \dim_k Z(B_{pw} )$  and    
$y_{pw}=  \dim_k \HH^r(B_{pw})$. We use the notation
\[ Z(t)  =\sum_{n=0}^{\infty} z_{pn}\, t^n\]
from the statement of Theorem \ref{HH-group-block} and we set
  \[ Y(t) =   \sum_{n=0}^{\infty}  \dim_k \HH^r(B_{pn})\, t^n=
  \sum_{n=0}^{\infty}  y_{pn}\, t^n. \]
Note that  by \cite[Theorem 7.2]{Chuang/Rouquier:2008a},   $y_{pw} $ is  the
dimension of the degree $r$ Hochschild cohomology of  any weight $w$ block of a  
symmetric group algebra.
Further, note that by the  Nakayama conjecture, proved by   
Brauer~\cite{Brauer:1947a} and Robinson~\cite{Robinson:1947a},  
the number of   weight
$w$  blocks  of $k\sym_n $ equals   $c(n-pw)  $.   Since the
Hochschild cohomology  of a finite dimensional algebra is the direct
sum of the Hochschild cohomologies of its  blocks, we  have that for
any  $s$ with  $0\leq s \leq p-1 $,  and any $n \geq 0$,
\[ \dim_k Z(k\sym_{pn+s} )= \sum_{w=0} ^{n}   z_{pw} c (p(n-w)
  +s), \] 
and 
\[ \dim_k\HH^r( k\sym_{pn +s})  = \sum_{w=0} ^{n}   y_{pw} c (p(n-w)
  +s). \]
In other words, for any $s$, $0\leq s \leq p-1 $, 
\begin{equation} \label{eq1} 
\sum_{n=0}^{\infty} \dim_k Z(k\sym_{pn+s} )\, t^{pn+s}  =t^s Z(t^p)
C_s(t^p)  
\end{equation}  and 
\begin{equation}\label{eq2}  
\sum_{n=0}^{\infty}  \dim_k\HH^r( k\sym_{pn +s})\, t^{pn+s}  =
t^sY(t^p) C_s(t^p).  
\end{equation}
Since  $B_0 \cong k $, we have $z_0= \dim_k Z(B_0) =  1$ 
and from this it follows that   $Z(t)$  is invertible in $\bZ[[t]]$.  
On the other hand,   we have
$y_0= \dim_k\HH^r(B_0) = 0 $ and $y_1=\dim_k\HH^r(B_p) \neq 0$ 
(cf. Remark \ref{rem2}).
 In particular,
$Y(t) $ is  divisible by $t $ but not by $t^2$
in $\bZ[[t]]$. So we may define a power series in $t$, with non-zero constant
term, 
\[ \phi(t)=(t^{-1}Y(t))Z(t)^{-1}\in\bZ[[t]], \] 
and then we have
\begin{equation} \label{eq3} 
Y (t)  = t \phi(t)  Z(t). 
\end{equation}
 Now, 
\begin{align*}    
\dim_k\HH^r( k\sym_n)\, t^{n}
&=  \sum_{s=0}^{p-1}  \sum_{n=0}^{\infty}   
\dim_k\HH^r( k\sym_{pn +s}) \, t^{pn+s} =  
Y(t^p)  \sum_{s =0}^{p-1} C_s(t^p)\, t^s  \\  
& =  t^p \phi(t^p)  \sum_{s =0}^{p-1} Z(t^p)C_s(t^p)  \, t^s 
\end{align*}
where we  have used~\eqref{eq2}  for the second equality 
and~\eqref{eq3}  for the  third equality.
On the  other hand,  by~\eqref{eq1},
\[  P(t)=  \sum_{s =0}^{p-1} \sum_{n=0}^{\infty}  
\dim_kZ( k\sym_{pn+s})\, t^{pn+s} 
=\sum_{s=0}^{p-1}  Z(t^p)  C_s(t^p) \,t^s. \]
So,  
\[ \sum_{n=0}^{\infty}  \dim_k\HH^r( k\sym_n) \, t^{n}=   t^p\phi(t^p)  P(t). \]
 Now by Theorem~1.3 of  \cite{Benson/Kessar/Linckelmann:bkl5}  we
 have that $t^p \phi(t^p) $ is a  rational  function of $t$.  It then 
follows from Lemma~\ref{p-subs} that $\phi(t)$ is a rational function
of $t$, and by the above, $\phi(0)$ is non-zero.
\end{proof}

\begin {proof}[Proof of Theorem~\ref{HH1-blockdim}]   
By  Theorem 1.2 of  \cite{Benson/Kessar/Linckelmann:bkl5}, 
\[ \displaystyle\sum_{n=0}^\infty \dim_{k}\HH^1(k\sym_n)\,t^n=
\begin{cases}
\displaystyle\frac{2t^2}{1-t^2}\,P(t) &p=2, \\[10pt]
\displaystyle\frac{t^p}{1-t^p}\,P(t) & p \ge 3.
\end{cases} \]

Note that in  \cite{Benson/Kessar/Linckelmann:bkl5}  the base field
is taken to be ${\mathbb F}_p $, but the  above result clearly  holds
as well for any field of characteristic $p$ (cf. Remark \ref{rem3}).  By
Theorem~\ref{HH-group-block}  and its notation,  applied with $r=1$, 
\[ \displaystyle\sum_{n=0}^\infty \dim_{k}\HH^1(B_{pn})\,t^n=
\begin{cases}
\displaystyle\frac{2t}{1-t}\, Z(t) &p=2, \\[10pt]
\displaystyle\frac{t}{1-t}\,Z(t) & p \ge 3.
\end{cases} \]

Here, note that $h(t) \to  h(t^m) $  is an injective map  on ${\mathbb Z} [[t]] $.
This, along with  \cite[Theorem 7.2]{Chuang/Rouquier:2008a}  proves  the  
equations
\[ \displaystyle \dim_{k}\HH^1(B_{pn})\, =
\begin{cases}
\displaystyle  2\sum_{j=0}^{w-1} \dim_k  Z(B_{pj} )
  &p=2, \\[10pt]
\displaystyle    \sum_{j=0}^{w-1} \dim_k  Z(B_{pj} )
& p \ge 3.
\end{cases} \]
of Theorem~\ref{HH1-blockdim}. The remaining  equations  follow
from the  Nakayama Conjecture, proved in~\cite{Brauer:1947a,Robinson:1947a}, 
which states that characters of a symmetric group belong to the same
$p$-block if and only if the corresponding partitions have the same
$p$-core, implying in particular that $\dim_k Z(B_{pj})=\rho(pj, \varnothing)$.
\end{proof}  

 Theorem~\ref{HH1-blocknonvanish}  is an immediate 
consequence of Theorem~\ref{HH1-blockdim}.

\bigskip\noindent
{\bf Acknowledgements.} 
The first  and second authors are grateful to City, University of
London for its hospitality during the research for this paper. 
The third author  acknowledges support from
EPSRC grant EP/T004592/1.

\bibliographystyle{amsplain}
\bibliography{../repcoh}
%\bibliography{repcoh}

\end{document}